\documentclass[letter,11pt]{amsart}

\usepackage{fullpage}			%margens menores
\usepackage[english]{babel}
\usepackage[utf8]{inputenc}		%poder escrever acentos no código
\usepackage[T1]{fontenc}		%uma fonte mais completa ao compilar (e.g. acentos sem isso são símbolos isolados)
\usepackage{amsmath}			%símbolos matemáticos
\usepackage{amssymb}			%mais símbolos matemáticos
\usepackage{amsthm}			%poder usar 'theoremstyle'
\usepackage{mathrsfs}			%\mathscr
\usepackage[pdftex,bookmarksnumbered]{hyperref}	%números no índice do leitor de PDF
\usepackage{eucal}
\usepackage{amsfonts}

\usepackage[numbers]{natbib}

\newcounter{generalnumbering} \numberwithin{generalnumbering}{section}

\theoremstyle{plain}		\newtheorem{theorem}[generalnumbering]{Theorem}
\theoremstyle{plain}		\newtheorem{corollary}[generalnumbering]{Corollary}
\theoremstyle{definition}		\newtheorem{definition}[generalnumbering]{Definition}
\theoremstyle{definition}		
\theoremstyle{plain}		\newtheorem{proposition}[generalnumbering]{Proposition}
\theoremstyle{plain}		\newtheorem{lemma}[generalnumbering]{Lemma}

\newenvironment{remark}
{\vspace{\topsep}\noindent\textbf{Remark.}}
{\vspace{\topsep}}

\author{Luiz Cordeiro}

\address{Institute of Mathematics and Statistics,
University of Ottawa,
585 King Edward Ave.,
Ottawa, ON K1N 6N5,
Canada}
\email{lcord081@uottawa.ca}

\title{An elementary approach to sofic equivalence relations}

\subjclass[2010]%
{Primary 
37A15;
Secondary
28D15,
47D03
}

\begin{document}

\maketitle

\begin{abstract}
We present an elementary description of sofic equivalence relations, as well as some permanence properties for soficity. We answer a question by Conley, Kechris and Tucker-Drob about determining soficity in terms of its full group.
\end{abstract}

\section*{Introduction}

The notion of soficity for groups was introduced by Gromov \cite{MR1694588} in his work of symbolic dynamics. In 2010, Elek and Lippner \cite{eleklippner2010} introduced the notion of soficity for equivalence relations in the same spirit as Gromov's original definition, i.e., an equivalence relation $R$, induced by some action of the free group $\mathbb{F}_\infty$, is sofic if the Schreier graph of the $\mathbb{F}_\infty$-space $X$ can be approximated, in a suitable sense, by Schreier graphs of finite $\mathbb{F}_\infty$-spaces.

Alternative definitions by Ozawa and P\v{a}unescu describe soficity at the level of the so-called full semigroup of $R$, or in terms of the action of the full group on the measure algebra. We describe general elementary techniques to deal with (abstract) sofic relations, in particular showing that these definitions are equivalent, and use them to prove that soficity is well-behaved with respect to countable decompositions of the space, finite-index extensions and products, as well as to some operations on the measure, namely direct integrals and substituting the measure by an equivalent one (so soficity can be seen as a property of a measure-class, instead of a specific measure)

\section{Definitions and notation}

A \emph{countable Borel} equivalence relation on a standard Borel space $X$ is an equivalence relation $R$ on $X$ which is Borel as a subset of the product space $X\times X$, and for each $x\in X$, the $R$-class $R(x)=\left\{y\in X:(x,y)\in R\right\}$ is countable. The \emph{Borel full semigroup} of $R$ is the set $[[R]]_B$ of partial Borel isomorphisms $g:\operatorname{dom}g\to\operatorname{ran}g$, where $\operatorname{dom}g$ and $\operatorname{ran}g$ are Borel subsets of $X$, for which $(x,gx)\in R$ for all $x\in\operatorname{dom}g$. $[[R]]_B$ is an inverse monoid\footnote{An \emph{inverse monoid} is a set $M$ with an associative binary operation $(x,y)\mapsto xy$, which has a neutral element $1$ and such that for each element $g\in M$ there is an unique element $h\in M$ satisfying $g=ghg$ and $h=hgh$, called the \emph{inverse} of $g$ and denoted $h=g^{-1}$.} with the usual composition of partial functions, i.e., for $g,h\in[[R]]_B$,
\begin{enumerate}
\item[(i)] $\operatorname{dom}(hg)=g^{-1}(\operatorname{ran}g\cap\operatorname{dom}h)$, $\operatorname{ran}(hg)=h(\operatorname{dom}h\cap\operatorname{ran}g)$
\item[(ii)] $(hg)(x)=h(g(x))$ for all $x\in\operatorname{dom}(hg)$.
\end{enumerate}

Now let $X$ be a standard Borel space and $\mu$ a Borel probability measure on $X$, in which case we call $(X,\mu)$ a \emph{standard probability space}. Let $R$ be a countable Borel equivalence relation on $X$. We say that $\mu$ is \emph{$R$-invariant}, or that $R$ is \emph{measure($\mu$)-preserving} if $\mu(g(A))=\mu(A)$ for all $g\in[[R]]_B$ and $A\subseteq\operatorname{dom}(g)$.

If $G$ is a countable group acting by measure-preserving Borel automorphisms on a standard probability space $(X,\mu)$, then the orbit equivalence relation $R_G=\left\{(x,gx):x\in X, g\in G\right\}$ is countable, Borel and probability measure-preserving. It is standard fact \cite{MR0578656} that every countable Borel probability measure-preserving relation $R$ is of the form $R=R_G$ for a certain countable group $G$ acting on $(X,\mu)$.

Throughout this paper, we will only consider stardard probability spaces, countable Borel measure-preserving equivalence relations, measure-preserving actions and countable groups, even when no explicit mention of these hypotheses is done.

An $R$-invariant measure $\mu$ induces a pseudometric $d_\mu$ on $[[R]]_B$ by
 \[d_\mu(g,h)=\mu(\operatorname{dom}g\triangle\operatorname{dom}h)+\mu\left\{x\in\operatorname{dom}g\cap\operatorname{dom}h:gx\neq hx\right\}.\]
The metric quotient is denoted $[[R]]_\mu$, or simply $[[R]]$ when $\mu$ is implicit, and is called the (measured) \emph{full semigroup} of $R$. It is a complete, separable inverse monoid with the naturally defined structure.

The \emph{trace} of an element $g\in[[R]]$ is $\operatorname{tr}g=\mu\left\{x\in\operatorname{dom}(g):gx=x\right\}$.

\begin{proposition}\label{propositionfullsemigroupisinverse}
Given $g,h,g',h'\in[[R]]_B$,
\begin{enumerate}
\item $d_\mu(gh,g'h')\leq d_\mu(g,g')+d_\mu(h,h')$.
%Simply note that $[ab\neq cd]\subseteq [b\neq d]\cup b^{-1}[a\neq c]$ and just take measures.\qedhere
\item $d_\mu(g,h^{-1})\leq d_\mu(g,ghg)+d_\mu(h,hgh)$.
\item $d_\mu(g,h)=d_\mu(g^{-1},h^{-1})$.
\end{enumerate}
\end{proposition}
%\begin{proof}
%First, verify the following equality:
%\[\operatorname{dom}(g)=\operatorname{dom}(g)\setminus\operatorname{dom}(ghg)\cup\left\{x\in\operatorname{dom}(ghg):gx\neq ghg x\right\}\cup\left\{x\in\operatorname{dom}(g)\cap\operatorname{ran}(h):gx=h^{-1}x\right),\]
%and the terms in the right-hand-side are disjoint. Taking the complement of the last in $\operatorname{dom}(u)$ and calculating its measure yields
%\begin{align*}
%&\mu(\operatorname{dom}(g)\setminus\operatorname{ran}(h))+\mu\left\{x\in\operatorname{dom}(g)\cap\operatorname{ran}(h):gx\neq h^{-1}x\right\}\\
%&\hspace{60pt}=\mu(\operatorname{dom}(g)\setminus\operatorname{dom}(ghg))+\mu\left\{x\in\operatorname{dom}(ghg):gx\neq ghg x\right\}=d(g,ghg)
%\end{align*}
%Doing the same with $h^{-1}$ in place of $g$ and $g^{-1}$ in place of $h$ yields
%\begin{align*}
%&\mu(\operatorname{ran}(h)\setminus\operatorname{dom}(g))+\mu\left\{x\in\operatorname{dom}(g)\cap\operatorname{ran}(h):gx\neq h^{-1}x\right\}\\
%&\hspace{60pt}=\mu(\operatorname{ran}(h)\setminus\operatorname{ran}(hgh))+\mu\left\{x\in\operatorname{ran}(hgh):h^{-1}x\neq(hgh)^{-1} x\right\}=d(h^{-1},(hgh)^{-1})\\
%&\hspace{60pt}=d(h,hgh)
%\end{align*}
%where the last equality follows from invariance of $\mu$. Adding both equalities above gives us the result.\qedhere
%\end{proof}

The \emph{Borel full group} $[R]_B$ of a countable Borel equivalence relation $R$ on $X$ is the set of those $g\in[[R]]_B$ with $\operatorname{dom}g=\operatorname{ran}g=X$. If $\mu$ is an $R$-invariant probability measure, the image of $[R]_B$ in $[[R]]$ is called the (measured) \emph{full group} $R$ and is denoted $[R]$.

The \emph{measure algebra} of a standard probability space $(X,\mu)$ is the set $\operatorname{MAlg}(X,\mu)$ of Borel subsets of $X$ modulo $\mu$-null sets, i.e., we identify Borel subsets $A,B\subseteq X$ when $\mu(A\triangle B)=0$, and this is also an inverse monoid under intersection. Given a $\mu$-preserving relation $R$, we can identify $\operatorname{MAlg}(X,\mu)$ as the set of idempotent of $[[R]]$, by sending (the class of) each $A\subseteq X$ to (the class of) the identity $1_A:A\to A$ of $A$.

Given $n\in\mathbb{N}$, denote $[n]=\left\{0,\ldots,n-1\right\}$ a set with $n$ elements, and consider the normalized counting measure $\mu_{\#,n}(A)=\#A/n$ on $[n]$. When no confusion arises, we simply write $\mu_\#$. By considering the full equivalence relation $R_n=[n]^2$ on $[n]$, its full semigroup is the set $[[n]]$ of all partial bijections of $[n]$, and the full group is simply the permutation group $\mathfrak{S}_n$. The metric associated with $\mu_\#$ is denoted $d_\#$, and called the normalized \emph{Hamming distance}, and the measure algebra (which consists of subsets of $[n]$) is denoted $\operatorname{MAlg}(n)$.

The language of metric ultraproducts is useful for soficity, and we'll describe them briefly here. We refer to \cite{pk12} and \cite{MR3408561} for the details. Let $(M_k,d_k)$ be a sequence of metric spaces of diameter $\leq 1$, and $\mathcal{U}$ a free ultrafilter on $\mathbb{N}$. The \emph{metric ultraproduct} of $(M_k,d_k)$ along $\mathcal{U}$ is the metric quotient of $\prod_k M_k$ under the pseudometric $d_\mathcal{U}((x_k),(y_k))=\lim_{n\to\mathcal{U}}d_k(x_k,y_k)$, and we denote it $\prod_{\mathcal{U}}M_k$. We denote the class of a sequence $(x_k)_k\in\prod_k M_k$ by $(x_k)_\mathcal{U}$.

We will be interested in ultraproducts of the semigroups $[[n]]$, $\operatorname{MAlg}(n)$ and $\mathfrak{S}_n$. We also extend the notion of domain, range, etc... to these ultraproducts, i.e., we consider maps
\[\operatorname{dom}:\prod_{\mathcal{U}}[[n_k]]\to\prod_{\mathcal{U}}\operatorname{MAlg}(n_k),\quad\operatorname{dom}(g_n)_\mathcal{U}=(\operatorname{dom}g_n)_\mathcal{U}\]
and similarly for $\operatorname{ran},\operatorname{supp},\operatorname{Fix}:\prod_{\mathcal{U}}[[n_k]]\to\prod_{\mathcal{U}}\operatorname{MAlg}(n_k)$ (respectively, range, support, and fixed points). The \emph{trace} on $\prod_{\mathcal{U}}[[n_k]]$ is given by
\[\operatorname{tr}(g_n)_\mathcal{U}=\lim_{n\to\mathcal{U}}\operatorname{tr}(g_n).\]
Moreover, by \ref{propositionfullsemigroupisinverse} $\prod_{\mathcal{U}}[[n_k]]$ is an inverse monoid with respect to the canonical product, namely $(g_n)_\mathcal{U}(h_n)_{\mathcal{U}}=(g_nh_n)_{\mathcal{U}}$. Also, the group $\prod_{\mathcal{U}}\mathfrak{S}_{n_k}$ acts on $\prod_{\mathcal{U}}\operatorname{MAlg}(n_k)$ via $(\sigma_k)_{\mathcal{U}}\cdot (A_k)_{\mathcal{U}}=(\sigma_k(A_k))_{\mathcal{U}}$.

If $f,g\in[[R]]$ coincide on the intersection of their domains, or equivalently $f^{-1}g,fg^{-1}$ are idempotents, we denote by $f\lor g\in[[R]]$ the map with $\operatorname{dom}(f\lor g)=\operatorname{dom}(f)\cup\operatorname{dom}(g)$, and which restricts to $f$ and $g$ on their respective domains. The same can also be defined in ultraproducts.

For a given $n$, we can identify $\mathfrak{S}_n$ with the group $P_n$ of permutation matrices, or more generally $[[n]]$ with the semigroup $Q_n$ of matrices formed by $0$'s and $1'$s, with at most one $1$ in each row and each column. These semigroups are respected by tensors and direct sums, i.e., if $A\in Q_n$ and $B\in Q_m$, then $A\otimes B\in Q_{n\times m}$ and $A\oplus B\in Q_{n+m}$. We translate these operations to $[[n]]$ and $[[m]]$: Given $f\in[[n]]$ and $g\in[[m]]$, $f\otimes g\in[[n\times m]]$ is given by $(f\otimes g)(i,j)=(f(i),g(j))$ for all $i,j$ for which this makes sense, and $f\oplus g\in[[n+m]]$ is given by $(f\oplus g)(i)=f(i)$ if $i\leq n$, and $(f\oplus g)(i)=g(i-n)+n$ if $n<i\leq n+m$.

One can avoid talking about ultraproducts as follows: Let $\prod[[n]]$ be endowed with the supremum metric and define an equivalence relation $\sim$ on $\prod[[n]]$ by setting $(x_n)\sim (y_n)$ if $\lim_{n\to\infty}d_\#(x_n,y_n)=0$. Denote by $\prod^{\ell^\infty/c_0}[[n]]=\prod_n[[n]]/\sim$ the quotient. Proposition \ref{propositionfullsemigroupisinverse} also implies that $\prod^{\ell^\infty/c_0}[[n]]$ is an inverse monoid with the obvious operations.

Each $[[n]]$ embeds into $[[n+1]]$ via $f\mapsto f\oplus 1$, and this changes the metric by at most $\frac{1}{n+1}$, and also $[[n]]$ embeds isometrically into $[[kn]]$ via $f\mapsto f\otimes 1_{[k]}$. This way, we can embed $[[n]]$ into any $p\geq n$ as follows: if $p=qn+r$, with $0\leq r<n$, embed $[[n]]$ into $[[qn]]$ and then into $[[qn+1]],[[qn+2]],\ldots,[[qn+r]$. The metric changes by at most $\frac{1}{qn+r}+\cdots+\frac{1}{qn+1}\leq\frac{n}{qn}=\frac{n}{p-r}\leq\frac{n}{p-n}$, and this goes to $0$ as $p\to\infty$. With these embeddings and a couple of diagonal arguments, one easily proves the following:

\begin{theorem}
A separable metric space (semigroup) $M$ embeds into $\prod_{\mathcal{U}} [[n_k]]$ if and only if $M$ embeds into $\prod^{\ell^\infty/c_0}[[n]]$.
\end{theorem}
%\begin{proof}
%One direction is easy. For the other, if $M$ embeds into the ultraproduct, then for each finite subset $F$ of $M$ and $\epsilon>0$ we can find a sequence $\sigma_{n,\epsilon}:F\to[[n]]$ which is $\epsilon$-isometric (and an $\epsilon$-morphism, if $M$ is a semigroup). applying a diagonal argument, we find a sequence $\sigma_n:F\to[[n]]$ such that $\lim_{n\to\infty}d_\#(\sigma_n(x),\sigma_n(y))=d(x,y)$ (and $\lim_{n\to\infty}d_\#(\sigma_n(xy),\sigma_n(x)\sigma_n(y))=0$, if $M$ is a semigroup). Apply a diagonal argument to a countable dense subset $D$ of $M$, and find $\sigma_n$ as above with $D$ in place of $F$. Then $\sigma=\prod_n\sigma_n$ is an embedding into $\prod^{\ell^\infty/c_0}[[n]]$ as we want.\qedhere
%\end{proof}

In particular, the choice of free ultrafilter $\mathcal{U}$ or of sequence $(n_k)$ does not matter for the existence of an embedding into $\prod_{\mathcal{U}}[[n_k]]$.

\section{Sofic equivalence relations}

We use a description of soficity by Ozawa.

\begin{definition}[{\cite{ozawasoficnotes}; \cite{MR3035288}}]
$R$ is \emph{sofic} if for each finite subset $K\subseteq[[R]]$ and each $\epsilon>0$, there exists $N\in\mathbb{N}$ and $\pi:[[R]]\to[[N]]$ satisfying:
\begin{enumerate}
\item[(i)] $\pi(\operatorname{id}_X)=1_{[N]}$; $\pi(\varnothing)=\varnothing$;
\item[(ii)] For all $\varphi,\psi\in K$, $d_\#(\pi(\varphi\psi),\pi(\varphi)\pi(\psi))<\epsilon$;
\item[(iii)] For all $\varphi\in K$, $|\mu(\left\{x:\varphi(x)=x\right\})-\mu_\#(\left\{m:\pi(\varphi)(m)=m\right\})|<\epsilon$.
\end{enumerate}
$\pi$ is called a \emph{$(K,\epsilon)$-almost morphism}.
\end{definition}

It is standard procedure to write this in terms of ultraproducts. In fact, condition (i) above is unnecessary.

\begin{theorem}
$R$ is sofic if and only if $[[R]]$ embeds isometrically in $\prod_{\mathcal{U}}[[n_k]]$. In fact, an embedding $\Phi:M\to\prod_{\mathcal{U}}[[n_k]]$ from any sub-inverse semigroup $M$ of $[[R]]$ containing $1_X$ is isometric if and only if it preserves the trace.
\end{theorem}
\begin{proof}[Sketch of proof]
The second assertion follows if we write the distance in terms of the trace and vice versa. First one verifies that if $\Phi$ is isometric then $\Phi(1_X)=1$, and then that
\[\operatorname{tr}(f)=1-d_\mu(1_{\operatorname{dom}(f)},1_X)-d_\mu(1_{\operatorname{dom}(f)},f).\]
Conversely,
\[d_\mu(f,g)=\operatorname{tr}(1_{\operatorname{dom}(f)})+\operatorname{tr}(1_{\operatorname{dom}(g)})-\operatorname{tr}(1_{\operatorname{dom}(f)}1_{\operatorname{dom}(g)})-\operatorname{tr}(f^{-1}g1_{\operatorname{dom}(f)}),\]
and analogous formulas hold in $\prod_{\mathcal{U}}[[n_k]]$.

For the first part, the definition of soficity allows us to isometrically embed a dense countable inverse semigroup of $[[R]]$ in $\prod_{\mathcal{U}}[[n_k]]$, and this extends to an embedding of $[[R]]$.\qedhere
\end{proof}

\begin{remark}\label{increasingunionofsoficequivalencerelations}
If $\left\{R_n\right\}_n$ is an increasing sequence of sofic equivalence relations, then $R=\bigcup_{n=1}^\infty R_n$ is also sofic. Indeed, $\left\{[[R_n]]\right\}_n$ is an increasing sequence of semigroups of $[[R]]$ with dense union, so almost morphisms of each $[[R_n]]$ give us the necessary almost morphisms of $[[R]]$.
\end{remark}

Next, we describe soficity in terms of the natural action of $[R]$ on $\operatorname{MAlg}(X,\mu)$. If $G$ and $H$ are groups acting on sets $X$ and $Y$, respectively, $\theta:G\to H$ is a homomorphism and $\phi:X\to Y$ is a function, we say that the pair $(\theta,\phi)$ is \emph{covariant} if it respects the respective group actions, i.e., if $\phi(g(x))=\theta(g)(\phi(x))$ for all $g\in G$ and $x\in X$.

\begin{lemma}\label{lemmacovariantembeddings}
Suppose $\phi:\operatorname{MAlg}(X,\mu)\to\prod_{\mathcal{U}}\operatorname{MAlg}(n_k)$ is an isometric embedding, $g\in[R]$ and $\sigma\in\prod_{\mathcal{U}}\mathfrak{S}_{n_k}$ satisfy $\sigma\cdot \phi(A)=\phi(g(A))$ for all $A\in\operatorname{MAlg}(X,\mu)$. Then $\operatorname{tr}(\sigma1_{\phi(A)})\leq\operatorname{tr}(g1_A)$ for all $A\in\operatorname{MAlg}(X,\mu)$. If $\operatorname{tr}(g)=\operatorname{tr}(\sigma)$ then we have equality.
\end{lemma}
\begin{proof}
Let  $A\in\operatorname{MAlg}(X,\mu)$. Given $\epsilon>0$, we can take a finite partition $\left\{B_1,\ldots, B_{n+1}\right\}$ of $A\cap\operatorname{supp}(g)$ for which $\mu(B_{n+1})<\epsilon$ and $g(B_i)\cap B_i=\varnothing$ for $1\leq i\leq n$. We then have $\sigma\cdot\phi(B_i)\cap\phi(B_i)=\varnothing$, so
\[\mu_\#(\phi(A)\cap\operatorname{supp}(\sigma))\geq\sum_{i=1}^n\mu_\#(B_i)>\mu(A\cap\operatorname{supp}(g))-\epsilon,\]
or equivalently $\operatorname{tr}(\sigma1_{\phi(A)})<\operatorname{tr}(g1_A)+\epsilon$. Letting $\epsilon\to 0$ gives us the desired inequality.

For the last assertion, apply the first part to $A$ and $X\setminus A$.\qedhere
\end{proof}

If $G$ is a countable group acting on $(X,\mu)$ and inducing a relation $R$, we identify each element of $G$ with its image in $[R]$. The trace of an element $g\in G$ is then $\operatorname{tr}(g)=\mu\left\{x\in X:gx=x\right\}$.

\begin{proposition}\label{propositionequivalencesofdefinitionofsoficity}
Let $R$ be a countable, Borel, probability measure-preserving equivalence relation on the standard probability space $(X,\mu)$. Let $G$ be a countable group acting on $X$ and inducing $R$. The following are equivalent:
\begin{enumerate}
\item[(1)] $R$ is sofic;
\item[(2)] There exist isometric embeddings $\theta:[R]\to\prod_{\mathcal{U}}\mathfrak{S}_{n_k}$ and $\phi:\operatorname{MAlg}(X,\mu)\to\prod_{\mathcal{U}}\operatorname{MAlg}(n_k)$ which form a covariant pair.
\item[(3)] There exist a trace-preserving homomorphism $\theta:G\to\prod_{\mathcal{U}}\mathfrak{S}_{n_k}$ and an isometric embedding $\phi:\operatorname{MAlg}(X,\mu)\to\prod_{\mathcal{U}}\operatorname{MAlg}(n_k)$ which form a covariant pair.
\end{enumerate}
Moreover, if $G$ acts freely ($\mu$-a.e.) on $X$, then $\theta$ in item 3.\ does not need to be trace-preserving in principle.
\end{proposition}
\begin{proof}
The last assertion follows from \ref{lemmacovariantembeddings}. Given an inverse semigroup $S$, denote by $E(S)$ the set of idempotents of $S$.

(1)$\Rightarrow$(2): If $R$ is sofic, consider an isometric embedding $\Phi:[[R]]\to\prod_{\mathcal{U}}[[n_k]]$, which we restrict to obtain isometric embeddings
\[\phi:\operatorname{MAlg}(X,\mu)=E([[R]])\to E(\prod_{\mathcal{U}}[[n_k]])=\prod_{\mathcal{U}}E([[n_k]])=\prod_{\mathcal{U}}\operatorname{MAlg}(n_k).\]
The actions of full groups on measure algebras are given by conjugation in full semigroups, from which follows that $(\theta,\phi)$ is covariant.

(2)$\Rightarrow$(3) is clear, by composing $\theta$ with the natural homomorphism from $G$ to $[R]$.

(3)$\Rightarrow$(1): Assume $(\theta,\phi)$ as in (3).

Suppose that $g\in[[R]]$ can be decomposed as a finite disjoint union $g=\bigvee_{n=1}^Ng_n1_{A_n}$, where $g_n\in G$ and the $A_n$ form a partition of $\operatorname{dom}(g)$. We define $\Phi(g)=\bigvee_{n=1}^N\theta(g_n)\phi(1_{A_n})\in\prod_{\mathcal{U}}[[n]]$. 

We show that $\Phi(g)$ does not depend on the decomposition $g=\bigvee_ng_n1_{A_n}$. Indeed, suppose $\bigvee_ng_n1_{A_n}=\bigvee_mh_m1_{B_m}$. Then $\bigvee_{n,m}g_n1_{A_n\cap B_m}=\bigvee_{n,m}h_m1_{A_n\cap B_m}$. For a fixed $n$, it is clear that $\theta(g)1_{\phi(A_n)}=\theta(g)\bigvee_m1_{\phi(A_n)\cap\phi(B_m)}=\bigvee_m\theta(g)1_{\phi(A_n)\cap\phi(B_m)}$, and similarly for $h$ and a fixed $m$.

Since $g_n1_{A_n\cap B_m}=h_m1_{A_n\cap B_m}$, in fact it suffices to prove that, for a given $g\in G$ and $A\in\operatorname{MAlg}(X,\mu)$, $g|_A=1_A$ implies $\theta(g)1_{\phi(A)}=1_{\phi(A)}$. Indeed, in this situation $\theta(g)1_{\phi(A)}$ has domain $\phi(A)$, and $\operatorname{tr}(\theta(g)1_{\phi(A)})=\operatorname{tr}(g1_A)=\mu(A)=\mu_\#(\phi(A))$ by Lemma \ref{lemmacovariantembeddings}, and this yields the result.

Moreover, the previous Lemma also readily implies that $\Phi$ is trace-preserving. It is easy enough to see that $\Phi$ preserves products, so $\Phi$ is a trace-preserving, hence isometric, morphism on the semigroup of those $g\in[[R]]$ which can be decomposed as $g=\bigvee_{n=1}^Ng_n1_{A_n}$ for $g_n\in G$, which is dense in $[[R]]$ and hence extends to an isometric embedding of $[[R]]$.\qedhere
\end{proof}

\begin{remark}
The description of soficity above is equivalent to the existence of a sofic embedding of the von Neumann algebra $vN(R)$ of $R$, as defined in \cite{paunescu2011}, in which it is proven that this coincides with the original definition of soficity by Elek and Lippner.
\end{remark}

\section{Permanence properties}

In this section we will be concerned with permanence properties of the class of sofic equivalence relations. When we need to specify the measure space $(X,\mu)$ for which an equivalence relation $R$ is sofic in the previously described sense, we will instead say that the \emph{system} $(X,\mu,R)$ is \emph{sofic}.

\begin{theorem}
Let $(X,\mu)$ be a standard probability space with a measure-preserving countable Borel equivalence relation $R$. Suppose $\mu$ has a disintegration of the form $\mu=\int_X p_xd\nu(x)$, where $\nu$-a.e. $p_x$ are $R$-invariant probability measures for which $(X,p_x,R)$ is sofic. Then $(X,\mu,R)$ is also sofic.

In particular, if a.e.\ ergodic component of $R$ is sofic, so is $(X,\mu,R)$.
\end{theorem}
\begin{proof}
Let's denote by $\operatorname{tr}_\mu$ the trace on $[[R]]_B$ with respect to $\mu$, and $\operatorname{tr}_x$ the trace with respect to $p_x$. For each $g\in[[R]]_B$,
\[\operatorname{tr}_\mu(g)=\int_X\operatorname{tr}_x(g)d\nu(x).\]

Let $K$ be a finite subset of $[[R]]_B$ and $\epsilon>0$. The maps $x\mapsto\operatorname{tr}_x(g)$, $g\in K$, take values in $[0,1]$, so by partitioning $[0,1]$ and taking preimages, we can find a finite partition $\left\{A_j\right\}_{j=1}^N$ of $X$ for which $|\operatorname{tr}_x(g)-\operatorname{tr}_y(g)|<\epsilon$ for all $g\in K$ whenever $x$ and $y$ belong to the same $A_j$. Now consider positive integers $M,p_1,\ldots,p_N$ such that
\[\sum_{j=1}^Np_j=M\qquad\text{and}\qquad|\nu(A_j)-\frac{p_j}{M}|<\frac{\epsilon}{N}\qquad\text{for all }j.\]

Fix elements $y_j\in A_j$ with $(X,p_{y_j},R)$ sofic, so we can take $(K,\epsilon)$-almost morphisms $\theta_j:[[R]]_B\to [[n_j]]$ for each $p_{y_j}$. Moreover, embedding all $[[n_j]]$ in the common semigroup $[[\prod_j n_j]]$, we can assume that all $n_j$ are equal to a unique $n$. Define $\theta(g)\in[[\sum_j n\times p_j]]=[[n\times M]]$ by
\[\theta(g)=\bigoplus_{j=1}^N(\theta_j(g)\otimes 1_{[p_j]})\]
Then for all $g\in K$,
\[\operatorname{tr}_{\#,nM}(\theta(g))=\frac{1}{Mn}\sum_{j=1}^Np_jn\operatorname{tr}_{\#,n}(\theta_j(g))=\frac{1}{M}\sum_{j=1}^N p_j\operatorname{tr}_{\#,n}(\theta_j(g)),\quad\text{and}\]
\begin{align*}
\operatorname{tr}_\mu(g)&=\sum_{j=1}^N\int_{A_j}\operatorname{tr}_x(g)d\nu(x)=\left(\sum_{j=1}^N\int_{A_j}\operatorname{tr}_{y_j}(g)d\nu(x)\right)\pm\epsilon=\left(\sum_{j=1}^N\nu(A_j)\operatorname{tr}_{y_j}(g)\right)\pm\epsilon\\
&=\sum_{j=1}^N(\frac{p_j}{M}\pm\frac{\epsilon}{N})\operatorname{tr}_{y_j}(g)\pm\epsilon=\sum_{j=1}^N\frac{p_j}{M}\operatorname{tr}_{y_j}(g)\pm 2\epsilon=\sum_{j=1}^N\frac{p_j}{M}\operatorname{tr}_{\#,n}(\theta_j(g))\pm3\epsilon\\
&=\operatorname{tr}_{\#,n}(\theta(g))\pm3\epsilon
\end{align*}
and for all $g,h\in K$,
\begin{align*}
d_{\#,nM}(\theta(g)\theta(h),\theta(gh))&=d_{\#,nM}(\bigoplus_{j=1}^N(\theta_j(g)\theta_j(h)\otimes 1_{p_j}),\bigoplus_{j=1}^N(\theta_j(gh)\otimes 1_{p_j}))\\
&=\frac{1}{Mn}\sum_{j=1}^n p_j nd_{\#,n}(\theta_j(g)\theta_j(h),\theta_j(gh))\\
&\leq\max_{1\leq j\leq N}d_{\#,n}(\theta_j(g)\theta_j(h),\theta_j(gh))<\epsilon.\qedhere
\end{align*}
\end{proof}

Given a non-null Borel subset $A$ of $X$, we denote by $\mu_A$ the normalized measure on $A$, i.e., $\mu_A(B)=\mu(B)/\mu(A)$ for $B\subseteq A$, by $R|_A=R\cap (A\times A)$ the restriction of $R$ to $(A,\mu_A)$, and by $\operatorname{tr}_A$ for the corresponding trace on $[[R|_A]]$.

\begin{proposition}\label{propositionsoficityforsubsets}
\begin{enumerate}
\item[(a)] If $R$ is sofic and $A\subseteq X$ is any (non-null) subset, then $R|_A$ is sofic.
\item[(b)] If $\left\{A_n\right\}$ is a countable Borel partition of $X$ by (non-null) $R$-invariant subsets, then $R$ is sofic if and only if each $R|_{A_n}$ is sofic.
\end{enumerate}
\end{proposition}
\begin{proof}
\begin{enumerate}
\item[(a)] Let $K\subseteq[[R|_A]]$ be a finite subset and $\epsilon>0$. Since $[[R|_A]]$ is contained in $[[R]]$ (as a semigroup, but with a different metric), there exists a $(K,\epsilon)$-almost morphism $\theta:[[R]]\to[[n]]$ for some $n\in\mathbb{N}$. We may assume that $1_A\in K$, and that $\theta(1_A)$ is an idempotent in $[[n]]$. For $g\in[[R|_A]]$, we have $1_Ag1_A=g$, so switching $\theta(g)$ by $\theta(1_A)\theta(g)\theta(1_A)$ if necessary, we can assume the range and domain of $\theta(g)$ are contained in $Y:=\operatorname{dom}(\theta(1_A))$. This defines a map $\theta_A:[[R|_A]]\to[[Y]]$.

To see that $\theta_A$ approximately preserves the trace, note that the trace on $[[R|_A]]$ and the trace on $[[Y]]$ are given respectively by
\[\operatorname{tr}_{A}(g)=\frac{\operatorname{tr}_{\mu}(g)}{\operatorname{tr}_{\mu}(1_A)}\qquad\text{and}\qquad\operatorname{tr}_{\#,Y}(\theta_A(g))=\frac{\operatorname{tr}_{\#,n}(\theta(g))}{\operatorname{tr}_{\#,n}(\theta(1_A))},\]
and these numbers are as close as necessary if $\epsilon$ is small enough. The distances are dealt with similarly, so $\theta_A$ approximately preserves products.
\item[(b)] Use the previous theorem with $\nu=\mu$ and $p_x(B)=\mu_{A_j}(B\cap A_j)$, where $A_j$ is the only element of the partition with $x\in A_j$.\qedhere
\end{enumerate}
\end{proof}

Now we will deal with finite-index subrelations, as defined in \cite{MR1007409}. Let $R$ and $S$ be countable Borel probability measure-preserving equivalence relations on $(X,\mu)$ with $R\subseteq S$. Then each $S$-class can be decomposed in (at most) countably many $R$-classes. For $x\in X$, we denote by  $J(x)$ the number of $R$-classes contained in $S(x)$, and note that $J:X\to\left\{1,2,\ldots,\infty\right\}$ is measurable and $S$-invariant.

\begin{definition}
$R$ is said to have \emph{finite index} in $S$ if $J(x)<\infty$ $\mu$-a.e.
\end{definition}

Let $Y$ be an $S$-invariant subset of $X$ on which $J$ is constant, say $J(x)=n$ a.e.\ on $Y$. Then there exist measurable maps $\psi_1,\ldots,\psi_n:Y\to Y$ such that for $\mu$-a.e.\ $x\in Y$, $\left\{R(\psi_i(x)):1\leq i\leq n\right\}$ is a partition of $S(x)$. The maps $\psi_1,\ldots,\psi_n$ are called \emph{choice functions} for $R\subseteq S$ (inside $Y$). Define a map $\sigma:S|_Y\to\mathfrak{S}_{n_k}$ by setting $\sigma(y,x)(i)=j$ if $\psi_i(x)R\psi_j(y)$. Then $\sigma$ is a 1-cocycle (i.e., a groupoid morphism).

We will say that $R\subseteq S$ admits \emph{invertible choice functions in $Y$} if there exists choice functions $\psi_1,\ldots,\psi_n$ for $R\subseteq S$ in $Y$ which are automorphisms. This is the case, for example, when $R|_Y$ is ergodic (\cite{MR1007409}, Lemma 1.3). Moreover, in this case we have $\psi_i\in[S]$.

Finally, we will say that $R\subseteq S$ admits invertible choice functions if it admits invertible choice functions in each set $Y_n=\left\{x\in X:J(x)=n\right\}$.

\begin{theorem}
Suppose $R\subseteq S$ is of finite index and admits invertible choice functions (e.g.\ $R$ is ergodic). If $R$ is sofic, so is $S$.
\end{theorem}
\begin{proof}
The sets $\left\{x\in X:J(x)=n\right\}$ are $S$-invariant and partition $X$, so \ref{propositionsoficityforsubsets} allows us to restrict to the case when the index is constant. Suppose that $\psi_1,\ldots,\psi_N$ are invertible choice functions for $R\subseteq S$ with associated cocycle $\sigma:S\to\mathfrak{S}_N$.

For each $f\in [[S]]$ and each pair $(i,j)\in[N]^2$, let $A_{f;j,i}=\left\{x\in\operatorname{dom}f:\sigma(f(x),x)(i)=j\right\}$. If $x\in A_j$, then $(\psi_i(x),\psi_j(f(x)))\in R$, which implies that $\psi_jf\psi_i^{-1}|_{\psi_i(A_{f;j,i})}\in[[S]]$.

Let $\Phi:[[R]]\to\prod_{\mathcal{U}}[[n_k]]$ be a sofic embedding. Denote by $E_{j,i}$ the usual matrix unit with $1$ in the $(i,j)$-th entry and 0 everywhere else (or rather the element of $[[n]]$ associated to it). Define $\Xi:[[S]]\to\prod_{\mathcal{U}}[[n_k\times N]]$ by
\[\Xi(f)=\bigvee_{i,j}\Phi(\psi_j f\psi_i^{-1}|_{\psi_i(A_{f;j,i})})\otimes E_{j,i}\]

First let's show that $\Xi$ is well-defined, i.e., that the terms in the right-hand side have disjoint domains and images: Suppose $(i,j)\neq(k,l)$. Then
\begin{align*}
&(\Phi(\psi_j f\psi_i^{-1}|_{\psi_i(A_{f;j,i})})\otimes E_{j,i})(\Phi(\psi_l f\psi_k^{-1}|_{\psi_k(A_{f;l,k})})\otimes E_{l,k})^{-1}\\
&\hspace{60pt}=\Phi(\psi_j f\psi_i^{-1}\psi_k f^{-1}\psi_l^{-1}|_{\left\{x\in\psi_l(f(A_{f;l,k})):\psi_k(f^{-1}(\psi_l^{-1}(x)))\in\psi_i(A_{f;j,i})\right\}})\otimes(E_{j,i}E_{k,l})
\end{align*}
If $i\neq k$, the second term above is zero. If $i=k$ but $j\neq l$, the domain of the map on which we are applying $\Phi$ becomes
\[\left\{x\in\psi_l(f(A_{f;l,i})):f^{-1}\psi_l^{-1}(x)\in A_{f;j,i}\right\}=\psi_l(f(A_{f;l,i}\cap A_{f;j,i}))=\varnothing.\]
This proves that the domains of the maps in the definition of $\Xi(f)$ are disjoint. The images are dealt with similarly, and so $\Xi$ is well-defined.

Now we need to show that $\Xi$ is a morphism. Suppose $f,g\in[[R]]$. We have
\begin{align*}
\Xi(f)\Xi(g)&=\bigvee_{i,j,k,l}\Phi(\psi_j f\psi_i^{-1}\psi_l g\psi_k^{-1}|_{\left\{x\in\psi_k(A_{g;l,k}):\psi_l g\psi_k^{-1}(x)\in\psi_i(A_{f;j,i})\right\}})\otimes E_{j,i}E_{l,k}\\
&=\bigvee_{i,j,k}\Phi(\psi_j fg\psi_k^{-1}|_{\left\{x\in\psi_k(A_{g;i,k}):g\psi_k^{-1}(x)\in A_{f;j,i}\right\}})\otimes E_{j,k}.
\end{align*}
This should be equal to $\Xi(fg)=\bigvee_{k,j}\Phi(\psi_j(fg)\psi_k^{-1}|_{\psi_k(A_{fg;j,k})})\otimes E_{j,k}$, so we need simply to show that for each $j$ and $k$,
\[\bigvee_i\left\{x\in\psi_k(A_{g;i,k}):g\psi_k^{-1}(x)\in A_{f;j,i}\right\}=\psi_k(A_{fg;j,k})\]
Let $x$ in the left-hand side, and let $y=\psi_k^{-1}(x)$, so for some $i$, $\sigma(gy,y)(k)=i$ and $\sigma(fgy,gy)(i)=j$, so
\[\sigma(fg y,y)(k)=\sigma(fgy,gy)\sigma(gy,y)(k)=\sigma(fgy,gy)(i)=j,\]
thus $y\in A_{gy;j,k}$, and $x=\psi_k(y)\in\psi_k(A_{fg;j,k}$.

For the converse inclusion, simply take $y=\psi_k^{-1}(x)$ again and $i=\sigma(gy,y)(k)$. Thus we've proved $\Xi$ is a morphism.

Finally, we need to show that $\Xi$ is trace-preserving. Note that
\[\operatorname{tr}\Xi(f)=\frac{1}{N}\sum_{i=1}^N\operatorname{tr}(\psi_i f\psi_i^{-1}|_{\psi_i(A_{f;i,i})}),\]
so we are done if we prove that $\operatorname{tr}(\psi_i f\psi_i^{-1}|_{\psi_i(A_{f;i,i})})=\operatorname{tr}(f)$. More specifically, let's show that $\left\{x\in\psi_i(A_{f;i,i}):\psi_i f\psi_i^{-1}(x)=x\right\}=\psi_i(\left\{y\in\operatorname{dom}f: fy=y\right\}$.

Let $x$ in the left-hand side, and let $y=\psi_i^{-1}x$. Then $fy=\psi_i^{-1} f\psi_i^{-1}x=\psi_i^{-1}x=y$. Conversely, suppose $y\in\operatorname{dom}f$ with $fy=y$, and let $x=\psi_i(y)$. $fy=y$ implies $\psi_i fy=\psi_i y$, i.e., $y\in A_{f;i,i}$, and also implies $\psi_i f\psi_i^{-1}(x)=x$, so $x$ is in the left-hand-side.

Finally, since $\psi_i\in[S]$ and $S$ is measure-preserving, we are done.\qedhere
\end{proof}

Recall that $R$ is \emph{periodic} if a.e.\ class of $R$ is finite, and \emph{aperiodic} is a.e.\ class of $R$ is infinite.

\begin{corollary}\label{periodicequivalencerelationsaresofic}
Each hyperfinite (amenable) equivalence relation $R$ is sofic.
\end{corollary}
\begin{proof}
If $R$ is periodic, the equality relation $I=\left\{(x,x):x\in X\right\}$ is sofic, has finite index in $R$, and it is easy to show that it admits invertible choice functions. For general hyperfinite relations, apply the previous case and the remark above Lemma \ref{lemmacovariantembeddings}.\qedhere
\end{proof}

\begin{theorem}
$(X,\mu,R)$ and $(Y,\nu,S)$ are sofic if and only if $(X\times Y,\mu\times\nu,R\times S)$ is sofic.
\end{theorem}
\begin{proof}
Let $G$ and $H$ be countable groups acting in a pmp way on $X$ and $Y$, respectively, and inducing the respective equivalence relations. Then $G\times H$ acts on $X\times Y$, via $(g,h)(x,y)=(gx,hy)$, and this action induces $R\times S$.

Take covariant sofic pairs $(\iota_R,\phi_R):(G,\operatorname{MAlg}(X,\mu))\to(\prod_{\mathcal{U}}\mathfrak{S}_{n_k},\operatorname{MAlg}(n_k))$ and $(\iota_S,\phi_S):(H,\operatorname{MAlg}(S,\nu))\to(\prod_{\mathcal{U}}\mathfrak{S}_{m_k},\operatorname{MAlg}(m_k))$.

We define $\kappa:G\times H\to\prod_{\mathcal{U}}\mathfrak{S}_{n_k\times m_k}$ by $\kappa(g,h)=\iota_R(g)\otimes\iota_S(h)$ and $\psi:\operatorname{MAlg}(X\times Y,\mu\times\nu)\to\prod_{\mathcal{U}}\operatorname{MAlg}(n_k\times m_k)$, defined on rectangles by $\psi(A\times B)=\phi_R(A)\times\phi_S(B)$, which extends uniquely to a semigroup embedding. Then $(\kappa,\psi)$ is tracial and covariant for $R\times S$.

For the converse, simply note that there is a canonical tracial embedding $T:[[R]]\to[[R\times S]]$, namely $T(g)(x,y)=(g(x),y)$, and similarly for $S$. Simply compose any sofic embedding of $[[R\times S]]$ with these to obtain sofic embeddings of $[[R]]$ and $[[S]]$.\qedhere
\end{proof}

\section{Soficity and full groups}

A well-known theorem of Dye \cite{MR0158048} states that when $R$ is aperiodic the full group $[R]$ completely determines $R$. With this in mind, we prove that $R$ is sofic if and only if $[R]$ embeds isometrically into $\prod_{\mathcal{U}}\mathfrak{S}_{n_k}$ in ``almost all cases'', namely when $R$ does not have singleton classes. This solves a question posed by Conley--Kechris--Tucker-Drob \cite{MR3035288} in this case.

\begin{lemma}
Let $\theta:[R]\to\prod_{\mathcal{U}}\mathfrak{S}_{n_k}$ be an isometric embedding. If $g,h\in[R]$ with $\operatorname{supp}g=\operatorname{Fix}h$, then $\operatorname{supp}(\theta(g))=\operatorname{Fix}(\theta(h))$.
\end{lemma}
\begin{proof}
Suppose $\operatorname{supp}(g)=\operatorname{Fix} h$. Then $d_\mu(g,h)=\operatorname{tr}(g)+\operatorname{tr}(h)=1$, so $d_\#(\theta(g),\theta(h))=\operatorname{tr}(\theta(g))+\operatorname{tr}(\theta(h))=1$, which implies
\[\mu_\#(\operatorname{Fix}(\theta(g))\cap\operatorname{Fix}(\theta(h))= 0\qquad\text{and}\qquad\mu_\#(\operatorname{Fix}(\theta(g)))+\mu_\#(\operatorname{Fix}(\theta(h))=1,\]
and this means that $\operatorname{Fix}(\theta(h))$ is the complement of $\operatorname{Fix}(\theta(g))$, i.e., $\operatorname{supp}(\theta(g))$.
\end{proof}

\begin{theorem}
An aperiodic, countable measure-preserving equivalence relation $R$ is sofic if and only if the full group $[R]$ embeds isometrically into an ultraproduct $\prod_{\mathcal{U}}\mathfrak{S}_{n_k}$.
\end{theorem}
\begin{proof}
Let $\theta:[R]\to\prod_{\mathcal{U}}\mathfrak{S}_{n_k}$ be an isometric embedding. We need to construct an embedding $\phi:\operatorname{MAlg}(X,\mu)\to\prod_{\mathcal{U}}\operatorname{MAlg}(n_k)$ for which the pair $(\phi,\theta)$ is covariant.

Given $A\in\operatorname{MAlg}(X,\mu)$, choose $g\in[R]$ with $\operatorname{supp}(g)=A$ (\cite{MR2583950}, Lemma 4.10). Consider a representative $\theta(g)=(g_k)_\mathcal{U}$, and define $\phi(A)=\operatorname{supp}(\theta(g))=(\operatorname{supp}(g_k))_\mathcal{U}$. We will show that $(\theta,\phi)$ satisfies the conditions of Proposition \ref{propositionequivalencesofdefinitionofsoficity}. We do this in steps, namely:
\begin{enumerate}
\item $\phi$ is well-defined, i.e., it does not depend on the choice of $g$ with $\operatorname{supp}(g)=A$;
\item $\phi$ preserves disjointness;
\item $\phi$ preserves intersections;
\item $\phi$ is covariant;
\item $\phi$ is isometric.
\end{enumerate}

\begin{enumerate}
\item\label{itemwelldefined} Suppose $g,h\in[R]$ with $\operatorname{supp}g=\operatorname{supp}h=A$. Consider any $r\in[R]$ with $\operatorname{supp}r=X\setminus A$. By the Lemma above, $\operatorname{supp}(\theta(g))=\operatorname{Fix}(\theta(r))=\operatorname{supp}(\theta(h))$.
\item\label{itempreservesdisjointness} Suppose $A\cap B=\varnothing$. Choose $g,h,r\in[R]$ with $\operatorname{supp}(g)=A$, $\operatorname{supp}(h)=B$, $\operatorname{supp}(r)=X\setminus(A\cup B)$. Here we consider representatives $\theta(g)=(g_k)_\mathcal{U}$, $\theta(h)=(h_k)_\mathcal{U}$, $\theta(r)=(r_k)_\mathcal{U}$. By the previous Lemma again, we can approximate, for $\mathcal{U}$-a.e.\ $k$, $\mu_\#(\operatorname{supp}(r_kh_k)\triangle\operatorname{Fix}(g_k))\sim 0$, so $h_k=r_k^{-1}$ in $\operatorname{supp}(g_k)\cap\operatorname{supp}(h_k)$ up to a set of measure $\sim 0$, and similarly $g_k=r_k^{-1}$ in $\operatorname{supp}(g_k)\cap\operatorname{supp}(h_k)$ up to a set of measure $\sim 0$. Thus
\begin{align*}
d_\#(g_k,h_k)&\sim\mu_\#(\operatorname{supp}(g_k)\triangle\operatorname{supp}(h_k))\\
&=\mu_\#(\operatorname{supp}(g_k))+\mu_\#(\operatorname{supp}(h_k))-2\mu_\#(\operatorname{supp}(g_k)\cap\operatorname{supp}(h_k))\\
&=d_\#(g_k,1)+d_\#(h_k,1)-2\mu_\#(\operatorname{supp}(g_k)\cap\operatorname{supp}(h_k).
\end{align*}
Taking the limit over $\mathcal{U}$, we have $d_\#(\theta(g),\theta(h))=d_\#(\theta(g),1)+d_\#(\theta(h),1)-2\mu_{\#}(\phi(A)\cap\phi(B))$, that is,
\[d_\mu(g,1)+d_\mu(h,1)=d_\mu(g,h)=d_\mu(g,1)+d_\mu(h,1)-2\mu_\#(\phi(A)\cap\phi(B)),\]
thus $\mu_\#(\phi(A)\cap\phi(B))=0$, which means that $\phi(A)\cap\phi(B)=\varnothing$.
\item Now, let's show that $\phi$ preserves intersections. Take $A,B\in\operatorname{MAlg}(X,\mu)$, and consider $g,h,k\in[R]$ with $\operatorname{supp}(g)=A\cap B$, $\operatorname{supp}(h)=A\setminus B$ and $\operatorname{supp}(k)=B\setminus A$. By item \ref{itempreservesdisjointness}, the supports of $\theta(g)$ and $\theta(h)$ are disjoint, so $\operatorname{supp}(\theta(gh))=\operatorname{supp}(\theta(g)\theta(h))=\operatorname{supp}(\theta(g))\cup\operatorname{supp}(\theta(h))$, and similarly for $g$ and $k$. Also, $\operatorname{supp}(gh)=A$, $\operatorname{supp}(gk)=B$, so by item \ref{itempreservesdisjointness} again,
\begin{align*}
\phi(A)\cap\phi(B)&=\operatorname{supp}(\theta(gh))\cap\operatorname{supp}(\theta(gk))\\
&=(\operatorname{supp}(\theta(g))\cup\operatorname{supp}(\theta(h)))\cap(\operatorname{supp}(\theta(g))\cup\operatorname{supp}(\theta(k)))\\
&=\operatorname{supp}(g)=\phi(A\cap B).
\end{align*}
\item To prove covariantness, let $A\in\operatorname{MAlg}(X,\mu)$ and $g\in[R]$. Take $h\in[R]$ with $\operatorname{supp}(h)=A$. Then $\operatorname{supp}(ghg^{-1})=g(A)$, and
\[\phi(g(A))=\operatorname{supp}(\theta(ghg^{-1}))=\operatorname{supp}(\theta(g)\theta(h)\theta(g)^{-1})=\theta(g)\cdot\operatorname{supp}(h)=\theta(g)\cdot\phi(A).\]
\item For the last property we simply need to show that $\phi$ preserves measure. Given $A=\operatorname{supp}(g)\in\operatorname{MAlg}(X,\mu)$, with $g\in[R]$, we have
\[\mu_\#(\phi(A))=\mu_\#(\operatorname{supp}(\theta(g)))=d_\#(\theta(g),1)=d_\mu(1,g)=\mu(\operatorname{supp}(g))=\mu(A).\qedhere\]
\end{enumerate}
\end{proof}

Now we extend this result to when $R$ has periodic points, but no singleton classes. Set $\operatorname{Per}_{\geq 2}(R)=\left\{x\in X:|R(x)|\geq 2\right\}$.

\begin{lemma}
There exists $\alpha\in [R]$ with $\operatorname{supp}\alpha=\operatorname{Per}_{\geq 2}(R)$.
\end{lemma}
\begin{proof}
This follows easily from the existence of a transversal for periodic relations (\cite{kechrisclassicaldescriptivesettheory}, Theorem 12.16).\qedhere
\end{proof}

\begin{lemma}
Let $\theta:[R]\to\prod_{\mathcal{U}}\mathfrak{S}_{n_k}$ be an isometric embedding and $g,h\in [R]$. Then $\operatorname{supp}g\cap\operatorname{supp}h=\varnothing$ if and only if $\operatorname{supp}\theta(g) \cap\operatorname{supp}\theta(h)=\varnothing$.
\end{lemma}
\begin{proof}
$\operatorname{supp}g\cap\operatorname{supp}h=\varnothing$ if and only if $d_\mu(g,h)=d_\mu(1,g)+d_\mu(1,h)$, and this condition is preserved by $\theta$.\qedhere
\end{proof}

\begin{theorem}
Suppose $R$ does not contain singleton classes. Then $R$ is sofic if and only if $[R]$ is metrically sofic.
\end{theorem}
\begin{proof}
Let $P=\operatorname{Per}_{\geq 2}(R)$ and $\operatorname{Aper}=X\setminus P$. By previous results, it is sufficient to show that $[R|_{\operatorname{Aper}}]$ is metrically sofic. Fix any $\alpha\in[R]$ with $\operatorname{supp}\alpha=P$. Let $\theta:[R]\to\prod_{\mathcal{U}}\mathfrak{S}_{n_k}$ be a tracial embedding. Let $(A_k)_\mathcal{U}=\operatorname{supp}\theta(\alpha)$.

For each $f\in[R|_{\operatorname{Aper}}]$, let $\widetilde{f}=f\lor 1_P$ be the natural extension of $f$ to $X$. By the previous Lemma, $\operatorname{supp}\theta(\widetilde{f})\cap\operatorname{supp}(\theta(\alpha))$ for all $f\in[R|_{\operatorname{Aper}}]$, so we can find a representative $\theta(\widetilde{f})=(\theta_k(f))_\mathcal{U}$ such that $\operatorname{supp}\theta_k(f)\cap A_k=\varnothing$ for all $n$, that is, $\operatorname{supp}\theta_k(f)\subseteq[n_k]\setminus A_k$.

Define $\eta:[R|_{\operatorname{Aper}}]\to\prod_{\mathcal{U}} \mathfrak{S}_{[n_k]\setminus A_k}$ by $\eta(f)=(\theta_k(f)|_{[n_k]\setminus A_k})_\mathcal{U}$. It is easy enough to see that this map is multiplicative, so it remains only to check that it is tracial. Given $f\in[R|_{\operatorname{Aper}}]$, one readily checks that
\[\operatorname{tr}f=(\operatorname{tr}\widetilde{f}-\mu(P))\mu(\operatorname{Aper}),\]
and similarly,
\[\operatorname{tr}_{[n_k]\setminus A_k}\theta_k(f)|_{[n_k]\setminus A_k}=\left(\operatorname{tr}\theta_k(f)-\frac{\# A_k}{n_k}\right)\left(1-\frac{\#A_k}{n_k}\right).\]
Now $\operatorname{tr}\theta_k(f)$ converges (along $\mathcal{U}$) to $\operatorname{tr}\widetilde{f}$, and $\#A_k/n_k$ converges to $\mu(\operatorname{supp}\alpha)=\mu(P)=1-\mu(\operatorname{Aper})$. Therefore $\eta$ is tracial.\qedhere
\end{proof}

\bibliographystyle{amsplain}
\bibliography{biblio}

\end{document}